\documentclass[a4paper,12pt]{amsart}
\usepackage{amsmath,amssymb,amsthm}
\usepackage[left=2cm,right=2cm]{geometry}
\usepackage{enumerate}
\usepackage{mathtools}
\usepackage{mathtools}
\usepackage{hyperref}

\makeatletter
\@namedef{subjclassname@2020}{%
	\textup{2020} Mathematics Subject Classification}
\makeatother

\newtheorem{thm}{Theorem}

\newtheorem{lem}[thm]{Lemma}
\newtheorem{cor}[thm]{Corollary}

\theoremstyle{definition}
\newtheorem{defn}[thm]{Definition}
\newtheorem{notation}[thm]{Notation}

\theoremstyle{remark}

\def\C{\mathbb C}
\def\R{\mathbb R}
\def\CC{\mathcal C}
\def\hol{\mathcal O}

\def\Cn{\mathbb{C}^n}

\def\Bn{\mathbb{B}_n}

\begin{document}
	
	\title[Cyclicity of Multipliers on the Unit Ball of $\mathbb{C}^n$: A Corona-Based Approach]
	{Cyclicity of Multipliers on the Unit Ball of $\mathbb{C}^n$: A Corona-Based Approach}
	
	\author{Pouriya Torkinejad Ziarati}
	
	\address{P.T.Ziarati\\
		Institut de Math\'ematiques de Toulouse; UMR5219 \\
		Universit\'e de Toulouse; CNRS \\
		UPS, F-31062 Toulouse Cedex 9, France} 
        \email{pouriya.torkinejad\_ziarati@math.univ-toulouse.fr}

	\begin{abstract}
             We study the cyclicity of multipliers in Dirichlet-type spaces \( D_\alpha(\mathbb{B}_n) \). Specifically, we show that a multiplier \( f \) analytic on a neighborhood of $\overline{\mathbb{B}}_n$, whose zero set on the unit sphere is a compact, smooth, complex tangential submanifold of real dimension \( m \leq n - 1 \), is cyclic in \( D_\alpha(\mathbb{B}_n) \) if and only if \( \alpha \leq \frac{2n - m}{2} \). Our approach combines classical results on peak sets in \( A^\infty(\mathbb{B}_n) \) due to Chaumat and Chollet with a Corona-type theorem with two generators for the multiplier algebra. 
        \end{abstract}

	
	
	
	\maketitle
\section{Introduction}
        
        Given a Banach space \( \mathcal{A} \) of holomorphic functions in a domain in \( \mathbb{C}^n \), suppose that polynomials are dense in \( \mathcal{A} \) and act as \emph{multipliers} - that is, for every polynomial \( p \), the product \( p h \in \mathcal{A} \) whenever \( h \in \mathcal{A} \). A function \( f \in \mathcal{A} \) is said to be \emph{cyclic} if the linear span of all polynomial multiples of \( f \) is dense in \( \mathcal{A} \). Equivalently, \( f \) is cyclic if there exists a sequence of polynomials \( p_n \) such that \( \| f p_n - 1 \|_{\mathcal{A}} \to 0 \) as \( n \to \infty \). An immediate consequence of this definition is that, on a fixed domain, cyclicity becomes harder to achieve as the norm increases. The problem of characterizing cyclic functions is highly dependent on the choice of space and remains nontrivial, even in classical Hilbert spaces of holomorphic functions on the unit disk. Complete classifications, such as Beurling's theorem for the Hardy space \cite{Beurling1949two}, are exceptional. The Brown--Shields conjecture, proposed in 1984 for the Dirichlet space, remains unresolved \cite{BS84}. Recent progress in cyclicity in spaces of holomorphic functions on the unit disk, particularly Dirichlet-type spaces, can be found in \cite{Beneteau2015_One_Variable}, \cite{ElFallah_BS_Conj}, \cite{Kellay_Cyc_2020}, and \cite{ELFALLAH2016_cyc}. A comprehensive textbook treatment of Dirichlet-type spaces is also available in \cite{ElFallah2014_book}.

        In several variables, a complete characterization of cyclic functions in Dirichlet-type spaces remains elusive, even for classical domains such as the unit ball and the polydisk. However, significant progress has been made in identifying cyclic polynomials, with explicit characterizations obtained in the unit ball of $\C^2$ and the bidisk \cite{Beneteau2016_bidisk}, \cite{Kosinski_Vavitsas2023}, \cite{Knese2019_Aniso_Bidisk}. 
        
        Here, we study Dirichlet-type spaces of the unit ball $D_\alpha(\Bn)$, $\alpha \in \R$ , endowed with
the norm $\|\cdot\|_\alpha$ (see Definition~\ref{def: Dirichlet Space}). 
Given $f\in\mathcal{O}(\overline{\mathbb{B}}_n)$, a holomorphic functions defined on a neighborhood of the closed unit ball, we want to know
for which $\alpha$ the function $f$ is cyclic in $D_\alpha(\Bn)$. 
Note that when $\alpha\le \beta$, $\|g\|_\alpha \le \|g\|_\beta$,
so as remarked above, if $f$ is cyclic in $D_\beta(\Bn)$, it is cyclic in $D_\alpha(\Bn)$ too.

        If $f \in \mathcal{O}(\overline{\mathbb{B}}_n)$ is non-vanishing in $\Bn$, it is known by \cite[Theorem~1]{VavitsasNonCyc} that its zero set on the boundary satisfies
        \begin{equation*}
            \mathcal{Z}(f) \cap \partial \Bn = \bigcup_{i=1}^k M_i,
        \end{equation*}
        where the $M_i$ are smooth complex tangential manifolds (see Definition~\ref{Defn: complex tangential}). If $\mathcal{Z}(f) \cap \partial \Bn$ is non-singular, then each $M_i$ is compact. The following theorem constitutes our main result and resolves the cyclicity problem for $f$ in this setting:
        \begin{thm} \label{MainThm}
            Let \( f \in \hol(\overline{\mathbb{B}}_n) \) be nonvanishing in $\Bn$, and suppose that $\mathcal{Z}(f) \cap \overline{\mathbb{B}}_n = M$, where $M$ is a compact complex tangential smooth manifold. Let \( m \coloneqq \dim_{\R} M \), if \( \alpha \leq \frac{2n - m}{2} \), then \( f \) is cyclic in \( D_{\alpha}(\mathbb{B}_n) \).
        \end{thm}
	A combination of Theorem~3 in \cite{VavitsasNonCyc}, which gives a sufficient condition for non-cyclicity, and Theorem~\ref{MainThm} leads to the following corollary.
 \begin{cor}
     Let \( f \) be as in Theorem~\ref{MainThm}, then \( f \) is cyclic in \( D_{\alpha}(\mathbb{B}_n) \) if and only if \( \alpha \leq \frac{2n - m}{2} \).
 \end{cor}
A consequence of Theorem~1.5 in a recent preprint by Knese, Pascoe, and Sola \cite{KnesePascoeSola2026_Stable} is that the zero set in the sphere of an irreducible polynomial, non vanishing on $\mathbb{B}_2$, is a union of finitely many smooth compact complex tangential manifolds. Therefore, our main result provides an alternative way to characterize cyclic polynomials in $\mathbb{B}_2$, previously obtained by Kosi\'nski and Vavitsas \cite{Kosinski_Vavitsas2023}. This follows from the fact that the product of two cyclic multipliers is cyclic (see page~13 of \cite{TorkinejadZiarati2025_PoletskyStessin} for a proof of this fact).
 
The proof of our main result relies on two main ingredients. The first is a result of Chaumat and Chollet (Theorem~\ref{Thm: PeakSet Chaumat Chollet}) on peak sets for the smooth ball algebra
\[
A^{\infty}(\Bn) := \hol(\Bn) \cap \mathcal{C}^{\infty}(\overline{\mathbb{B}}_n)
\]
(see Definition~\ref{Def: peak set}). This result allows us to construct a holomorphic function with the desired zero set together with sharp estimates to determine cyclicity. The second tool is the following result. An analogous version in one variable was proved in \cite{Kellay2024_Corona_Cyclicity} using Tolokonnikov's Corona Theorem \cite{Tolokonnikov_Corona}. We follow the same strategy, using Theorem~\ref{Mini_Corona} to obtain the result.

\begin{thm} \label{Subspace_Inclusion}
    Let $f,g \in A^{\infty}(\Bn)$ and assume that on $\Bn$ we have $|g| \leq |f|$ then there exists a natural number $N$ such that:
    \begin{equation}
        [g^N] \subseteq [f]
    \end{equation}
    Here, by $[f]$ we mean the closed subspace of $D_\alpha(\Bn)$ generated by $f$.
\end{thm}
After completion of this work, Stefan Richter kindly informed us that a more general statement of Theorem~{\ref{Subspace_Inclusion}} was already obtained in \cite[Theorem~{3.2}]{AlemanPerfektRichterSundbergSunkes} using a different method.

The organization of the paper is as follows. In Section~\ref{Section2}, we develop the necessary tools to prove the main result. In particular, Subsection~\ref{Subsec2.1} introduces the notation and basic definitions; Subsection~\ref{Subsec2.2} is devoted to the proof of Theorem~\ref{Subspace_Inclusion}; and Subsection~\ref{Subsec2.3} presents a result of Chaumat and Chollet concerning peak sets of \( A^\infty(\Bn) \). In Section~\ref{Section3}, we provide the proof of our main result, Theorem~\ref{MainThm}.

\section{Preliminary Tools} \label{Section2}
\subsection{Basics and Notations} \label{Subsec2.1}
\begin{notation}
    Let $\Omega \subseteq \mathbb{C}^n$ be a domain. We denote by $\mathcal{O}(\Omega)$ the set of all holomorphic functions \( f : \Omega \to \mathbb{C} \).
\end{notation}
\begin{notation}
    We write $A \lesssim B$ to indicate that there exists a constant $C > 0$ such that $A \leq C B$. If both $A \lesssim B$ and $B \lesssim A$ hold, we write $A \asymp B$. The dependence of the constant $C$ on parameters will be specified whenever necessary.
\end{notation}

\begin{defn}
Let $z_0 \in \Cn$ and $E \subset \Cn$. We define
\begin{equation*}
    \mathrm{dist}(z_0,E)=\inf_{w\in E}|z_0-w|.
\end{equation*}
\end{defn}

\begin{defn} \label{def: Dirichlet Space}
    Let \( \alpha \in \mathbb{R} \). A function \( f \in \mathcal{O}(\mathbb{B}_n) \), with power series expansion
    \begin{equation*}
        f(z) = \sum_{|k| = 0}^{\infty} a_k z^k, \quad z \in \mathbb{B}_n,
    \end{equation*}
    is said to belong to the Dirichlet-type space \( D_\alpha(\mathbb{B}_n) \) if
    \begin{equation} \label{eq: Dirichlet type norm}
        \|f\|_{\alpha}^2 := \sum_{|k| = 0}^{\infty} (n + |k|)^\alpha \cdot \frac{ (n - 1)!\, k!}{ (n - 1 + |k|)!} \, |a_k|^2 < \infty,
    \end{equation}
    where $k$ is a multi-index, \( |k| = k_1 + \cdots + k_n \), \( k! = k_1! \cdots k_n! \), and  \( z^k = z_1^{k_1} \cdots z_n^{k_n} \).
\end{defn}
In particular, when \( \alpha = 0 \),  \( D_0(\mathbb{B}_n) = H^2(\mathbb{B}_n) \), the Hardy space. We will systematically use 
the characterization of Dirichlet-type spaces as generalized weighted Bergman spaces. Recall the classical definition
of weighted Bergman spaces:
\begin{defn}
Let $dv$ denote the $2n$-dimensional Lebesgue measure on $\C^n$.
A function $f \in \mathcal{O}(\mathbb{B}_n)$ belongs to the weighted Bergman space $A^2_\beta(\Bn)$, where $\beta > -1$, if
\begin{equation*}
    \| f \|^2_{A^2_{\beta}(\Bn)}
    =
    \int_{\Bn} (1-|z|^2)^{\beta} |f(z)|^2 \, dv(z)
    < \infty.
\end{equation*}
\end{defn}
This definition was extended to all $\beta\in\R$ by Zhao and Zhu \cite{ZhuZhao}.
\begin{defn} \label{Def: Radial Derivative}
\begin{enumerate}
\item
    Given $f \in \mathcal{O}(\Bn)$, the \emph{radial derivative} of $f$, denoted by $Rf$, is defined as follows:
    \begin{equation*}
        R f (z) = \sum_{i = 1}^n z_i \frac{\partial f}{\partial z_i} (z).
    \end{equation*}
\item
Given $\beta\in\R$ and $k\in \mathbb Z_+$ such that $\beta + 2 k >-1$, 
 \begin{equation*}
        f \in A^2_{\beta}(\Bn) \mbox{ if and only if } \int_{\Bn} (1-|z|^2)^{\beta+2k} |R^k f(z)|^2 \, dv(z)
    < \infty.
    \end{equation*}
\end{enumerate}
\end{defn}
This definition is consistent with the classical one and does not depend on the choice of $k$. 
It is an important fact that $D_{\alpha}(\Bn)=A^2_{-(\alpha+1)}(\Bn)$.
(Compare the power series norm in \eqref{eq: Dirichlet type norm} with (20) in \cite{ZhuZhao}.)
As a direct consequence we have:
\begin{lem} \label{RLem}
    $f \in D_{\alpha}(\Bn)$ if and only if $R^k f \in D_{\alpha - 2k}(\Bn)$. Also whenever $\alpha<0$ we have:
    \begin{equation} \label{EqvBergNorm}
        \| f \|^2_\alpha \simeq \int_{\Bn} (1-|z|^2)^{-(\alpha+1)} |f(z)|^2 \, dv(z)
    \end{equation}
\end{lem}

\begin{defn}
    Let \( \alpha \in \mathbb{R} \). The \emph{multiplier space} \( \mathcal{M}(D_\alpha(\Bn)) \) is the set of all functions \( \varphi \in D_\alpha(\mathbb{B}_n) \) such that for every \( f \in D_\alpha(\mathbb{B}_n) \), the product \( \varphi f \in D_\alpha(\mathbb{B}_n) \).

    The space \( \mathcal{M}(D_\alpha(\Bn)) \) is a Banach algebra when equipped with the operator norm induced by multiplication. In the
    well-known case of the Hardy space, we have
    \[
        \mathcal{M}(D_0(\mathbb{B}_n)) = H^\infty(\mathbb{B}_n),
    \]
    the space of bounded holomorphic functions on \( \mathbb{B}_n \).
\end{defn}

As a corollary of Lemma \ref{RLem}, we have an easy sufficient condition to be a multiplier.
\begin{cor} \label{MultiplierCor}
    Let \( \alpha \in \mathbb{R} \), and let \( k \in \mathbb{N} \) be such that \( \alpha - 2k < 0 \). If $f \in D_\alpha(\Bn)$ and \( R^k f \in H^\infty(\mathbb{B}_n) \), then \( f \in \mathcal{M}(D_\alpha(\Bn)) \). In particular, $\mathcal{O}(\overline{\mathbb{B}}_n) \subset A^{\infty}(\Bn) \subset \mathcal{M}(D_\alpha(\Bn))$.
\end{cor}
\begin{proof}
    Given $h \in \hol(\Bn)$ we have
    \begin{equation*}
        \left(\int_{\Bn} |h(z)|^2 \, dv(z) \right)^{\frac{1}{2}} \leq |h(0)| + \left(\int_{\Bn} |Rh(z)|^2 \, dv(z) \right)^{\frac{1}{2}}. 
    \end{equation*}
    Therefore, for $g \in D_\alpha(\Bn)$ we may write
    \begin{equation*}
    \begin{aligned}
        \| fg \|_\alpha 
        &\asymp |f(0)g(0)| 
        + \left( \int_{\Bn} (1-|z|^2)^{-(\alpha - 2k +1)}
        \left|R^{k} (fg)(z) \right|^2 \, dv(z) \right)^{\frac{1}{2}} \\
        &\leq C_k \sup_{0\leq l\leq k} \|R^l f\|_{L^\infty(\Bn)}
        \left(
        |g(0)| 
        + \sum_{l' = 0}^k \left( \int_{\Bn} (1-|z|^2)^{-(\alpha - 2k +1)}
        \left|R^{l'} g(z) \right|^2 \, dv(z) \right)^{\frac{1}{2}}
        \right) \\
        &\leq C'_k \sup_{0\leq l\leq k} \|R^l f\|_{L^\infty(\Bn)}
        \left(
        |g(0)| 
        + \left( \int_{\Bn} (1-|z|^2)^{-(\alpha - 2k +1)}
        \left|R^{k} g(z) \right|^2 \, dv(z) \right)^{\frac{1}{2}}
        \right) \\
        &\leq C''_k \sup_{0\leq l\leq k} \|R^l f\|_{L^\infty(\Bn)}
         \| g \|_\alpha 
         < \infty.
    \end{aligned}
    \end{equation*}
    So, $fg \in D_\alpha(\Bn)$ and the proof is done.
\end{proof}

\begin{defn}
The invariant (under the coordinate multiplication operators) subspace generated by $f$ is
\begin{equation*}
        [f] := \overline{\operatorname{span} \{ p f : p \in \mathbb{C}[z_1,\ldots,z_n] \}} .
\end{equation*}
    A function $f \in D_\alpha(\Bn)$ is called  \emph{cyclic}, or a cyclic vector, if
    \( [f]  = D_\alpha(\mathbb{B}_n).   \)
\end{defn}

The next theorem provides a practical sufficient condition for the cyclicity of a function. For a detailed proof, one may consult \cite{Knese2019_Aniso_Bidisk}.

\begin{thm} \label{Radial_Dilation}
    Let $f \in D_{\alpha}(\Bn)$ and assume that $f$ does not vanish on $\Bn$. For $0 < r < 1$, define the radial dilation
    \begin{equation*}
        f_r (z) = f(rz).
    \end{equation*}
    If
    \begin{equation*}
        \sup_{0 < r < 1} \left\| \frac{f}{f_r} \right\|_{\alpha} < +\infty,
    \end{equation*}
    then $f$ is cyclic in $D_{\alpha}(\Bn)$.
\end{thm}

\subsection{Proof of Theorem~\ref{Subspace_Inclusion}} \label{Subsec2.2}

First, we need the following classical result on the solution of the $\bar{\partial}$-problem by Lieb and Range \cite{Lieb_Range_dbar}.
\begin{thm} \label{dbar_thm}
    Let \( \Omega \subset\subset \mathbb{C}^n \) be a strictly pseudoconvex domain with \( \CC^{l+2} \) smooth boundary. Then for each \( q = 1, 2, \ldots, n \), there exist linear operators
    \begin{equation*}
        T_q^* : \CC_{(0,q)}^\infty(\overline{\Omega}) \to \CC_{(0,q-1)}^\infty(\overline{\Omega})
    \end{equation*}
    with the following properties:

    \begin{enumerate}[(i)]
        \item If \( f \in \CC_{(0,q)}^\infty(\overline{\Omega}) \) and \( \bar{\partial} f = 0 \), then
        \begin{equation*}
            \bar{\partial} T_q^* f = f.
        \end{equation*}

        \item[(ii)] If \( f \in \CC_{(0,q)}^l(\overline{\Omega}) \) with \( \bar{\partial} f = 0 \), then for all \( l = 0, 1, 2, \ldots \),
        \begin{equation*}
            \| T_q^* f \|_{\CC^{l , 1/2}(\overline{\Omega})} \leq C_l \| f \|_{\CC^l(\overline{\Omega})}.
        \end{equation*}
    \end{enumerate}
    The constants \( C_l \) depend only on \( l \) and \( \Omega \).
    
    Here, \( \| \cdot \|_{\CC^{l,1/2}(\overline{\Omega})} \) and \( \| \cdot \|_{\CC^l(\overline{\Omega})} \) denote the Hölder norms, obtained by:
    \begin{equation*}
        \|f\|_{{\CC^l(\overline{\Omega})}} = \sum_{|\alpha| \leq l} \sup_{x \in \overline{\Omega}} |D^\alpha f(x)|,
    \end{equation*}
    \begin{equation*}
        \|f\|_{\CC^{l,\varepsilon}(\overline{\Omega})}
        = \|f\|_{\CC^l(\overline{\Omega})}
        + \sum_{|\alpha| = l}\ \sup_{x \neq y}\ 
        \frac{|D^\alpha f(x) - D^\alpha f(y)|}{|x - y|^\varepsilon},
        \qquad 0<\varepsilon<1.
    \end{equation*}
\end{thm}

\begin{defn}
    $A^{l,\varepsilon}(\Bn) = \mathcal{O}(\Bn) \cap \mathcal{C}^{l , \varepsilon}(\overline{\mathbb{B}}_n)$. 
\end{defn}

\begin{thm} \label{Mini_Corona}
    Let $f_1,f_2 \in A^{\infty}(\Bn)$ and $\delta > 0$ be such that $|f_1|^2 + |f_2|^2 \geq \delta$ on $\Bn$. Then, for any $\alpha \in \mathbb{R}$, there exist $g_1,g_2 \in \mathcal{M}(D_\alpha(\Bn))$ and a natural number $k$ satisfying:
    \begin{enumerate} [(i)]
        \item \label{Corona_Property} $f_1 g_1 + f_2 g_2 = 1$.
        \item \label{Corona_Norm_Estimate}
        $\| g_i \|_{\mathcal{M}(D_\alpha(\Bn))} \leq \frac{C}{\delta^k}
        \| f_1 \|_{\mathcal{C}^{k}(\overline{\mathbb{B}}_n)}^{k}
    \| f_2 \|_{\mathcal{C}^{k}(\overline{\mathbb{B}}_n)}^{k}$.
    \end{enumerate}
    for some constant $C>0$ dependent on $k$.
\end{thm}
\begin{proof}
    Let $g_i^0 = \frac{\bar{f}_i}{|f_1|^2 + |f_2|^2}$ for $i = 1,2$. Then $(g_1^0, g_2^0)$ satisfies \eqref{Corona_Property}. Note that
    \begin{equation*}
        \bar{\partial} g_1^0 = f_2 \omega, \quad \bar{\partial} g_2^0 = -f_1 \omega,
    \end{equation*}
    where
    \begin{equation*}
        \omega = \frac{\bar{f}_2 \, \bar{\partial} \bar{f}_1 - \bar{f}_1 \, \bar{\partial} \bar{f}_2}{(|f_1|^2 + |f_2|^2)^2}.
    \end{equation*}

    Moreover, a direct computation shows that $\omega$ is a closed $(0,1)$-form and smooth up to the boundary.

    Let $A = \bar{f}_2 \, \bar{\partial} \bar{f}_1 - \bar{f}_1 \, \bar{\partial} \bar{f}_2$, 
$B = |f_1|^2 + |f_2|^2$, and let $X_1 X_2 \ldots X_l$ be a differential operator of order $l$. 
For any $\{l_1,\ldots,l_m\} \subseteq \{1,\ldots,l\}$, we have
\begin{equation*}
    |X_{l_1} \ldots X_{l_m} A|
    \leq C_m \| f_1 \|_{\mathcal{C}^{m+1}(\overline{\mathbb{B}}_n)}
    \| f_2 \|_{\mathcal{C}^{m+1}(\overline{\mathbb{B}}_n)}
    \leq C_m \| f_1 \|_{\mathcal{C}^{l+1}(\overline{\mathbb{B}}_n)}
    \| f_2 \|_{\mathcal{C}^{l+1}(\overline{\mathbb{B}}_n)},
\end{equation*}
and
\begin{equation*}
    \left|X_{l_1} \ldots X_{l_m}\!\left(\frac{1}{B^2}\right)\right|
    \leq C'_m \left|\frac{1}{B^{m+2}}\right|
    \| B \|_{\mathcal{C}^{m}(\overline{\mathbb{B}}_n)}^{m}
    \leq C''_m \frac{1}{\delta^{m+2}}
    \| f_1 \|_{\mathcal{C}^{m}(\overline{\mathbb{B}}_n)}^{m}
    \| f_2 \|_{\mathcal{C}^{m}(\overline{\mathbb{B}}_n)}^{m}.
\end{equation*}
    This implies 
    \begin{equation*}
        \| \omega \|_{{\mathcal{C}^{l}(\overline{\mathbb{B}}_n)}} \lesssim \frac{1}{\delta^{l+2}} \| f_1 \|_{\mathcal{C}^{l+1}(\overline{\mathbb{B}}_n)}^{l+1}
    \| f_2 \|_{\mathcal{C}^{l+1}(\overline{\mathbb{B}}_n)}^{l+1}.
    \end{equation*}
    By Theorem \ref{dbar_thm}, there exists a solution $\phi_l$ to the problem $\bar{\partial} \phi_l = \omega$ such that
    \begin{equation*}
        \| \phi_l \|_{{\mathcal{C}^{l,\frac{1}{2}}(\overline{\mathbb{B}}_n)}} \lesssim \frac{1}{\delta^{l+2}} \| f_1 \|_{\mathcal{C}^{l+1}(\overline{\mathbb{B}}_n)}^{l+1}
    \| f_2 \|_{\mathcal{C}^{l+1}(\overline{\mathbb{B}}_n)}^{l+1}.
    \end{equation*}

    Now, let $k \in \mathbb{N}$ be such that $\alpha - 2k <0$, we claim that the choice
    \begin{equation*}
        g_1 = g_1^0 - f_2 \phi_k, \quad g_2 = g_2^0 + f_1 \phi_k
    \end{equation*}
    is the desired solution. By construction, it satisfies \eqref{Corona_Property}. Moreover, by Lemma \ref{RLem}, a sufficient condition for $g_1$ and $g_2$ to be in the multiplier space of $D_\alpha(\mathbb{B}_n)$ is that their radial derivatives up to order $k$ are $L^\infty$-bounded. Hence, the norm estimate in \eqref{Corona_Norm_Estimate} follows.
\end{proof}

With the necessary tools in hand, we are now ready to achieve the main goal of this subsection:

\begin{proof} [Proof of Theorem~\ref{Subspace_Inclusion}]
    For $\lambda \in \C$, define
    \begin{equation*}
        \delta(\lambda) = \inf_{z \in \Bn} \left\{ |1 - \lambda g(z)|^2 + |f(z)|^2 \right\}.
    \end{equation*}
    If $|g(z)| \leq \frac{1}{2|\lambda|}$, then $|1 - \lambda g(z)|^2 \geq \frac{1}{4}$. On the other hand, whenever $|g(z)| \geq \frac{1}{2|\lambda|}$, we have $|f(z)| \geq \frac{1}{2|\lambda|}$. Therefore,
    \begin{equation*}
        \delta(\lambda) \geq \min\left\{ \frac{1}{4|\lambda|^2}, \frac{1}{4} \right\}.
    \end{equation*}

    By Theorem \ref{Mini_Corona}, there exist $g_1, g_2 \in \mathcal{M}(D_\alpha(\Bn))$ such that:
    \begin{enumerate}[(A)]
        \item \label{Corona_Solution} $g_1 (1 - \lambda g) + g_2 f \equiv 1$,
        \item $\|g_i\|_{\mathcal{M}(D_\alpha(\Bn))} \lesssim \frac{1}{\delta(\lambda)^k} \| 1- \lambda g\|^k_{\mathcal{C}^{k}(\overline{\mathbb{B}}_n)} \|f\|^k_{\mathcal{C}^{k}(\overline{\mathbb{B}}_n)} \lesssim |\lambda|^{3k}$ for $i = 1,2$.
    \end{enumerate}

    Now let $[f]_{\mathcal{M}(D_\alpha(\Bn))}$ be the closed ideal generated by $f$ in ${\mathcal{M}(D_\alpha(\Bn))}$. Define the standard quotient map
    \begin{equation*}
        \pi : \mathcal{M}(D_\alpha(\Bn)) \rightarrow \mathcal{M}(D_\alpha(\Bn)) / [f]_{\mathcal{M}(D_\alpha(\Bn))}.
    \end{equation*}
    Then by \eqref{Corona_Solution},
    \begin{equation*}
        \pi\left(1 - \lambda g\right)^{-1} = \pi(g_1).
    \end{equation*}

    Thus, \( \pi\left(1 - \lambda g\right)^{-1} \), which is a holomorphic Banach-valued function of \( \lambda \), is defined for every \( \lambda \), so it is an entire function. Hence, we may write:
    \begin{equation}\label{eq:whatever}
        \begin{aligned}
        \Bigl\| \pi\left(1-\lambda g\right)^{-1} \Bigr\|_{\mathcal{M}(D_\alpha(\Bn)) / [f]_{\mathcal{M}(D_\alpha(\Bn))}}
        &= \Bigl\| \pi( \sum_{j \ge 0} \lambda^j g^j) \Bigr\|_{\mathcal{M}(D_\alpha(\Bn)) / [f]_{\mathcal{M}(D_\alpha(\Bn))}} \\
        &= \Bigl\| \sum_{j \ge 0} \lambda^j \pi(g^j) \Bigr\|_{\mathcal{M}(D_\alpha(\Bn)) / [f]_{\mathcal{M}(D_\alpha(\Bn))}} \\
        &= \|\pi(g_1)\|_{\mathcal{M}(D_\alpha(\Bn)) / [f]_{\mathcal{M}(D_\alpha(\Bn))}} \\
        &\le \|g_1\|_{\mathcal{M}(D_\alpha(\Bn))} \lesssim |\lambda|^{3k}.
        \end{aligned}
    \end{equation}

    Now, by Liouville's Theorem, we conclude that $\pi(g^{3k+1}) = 0$. In other words,
    \begin{equation*}
        [g^{3k+1}]_{\mathcal{M}(D_\alpha(\Bn))} \subseteq [f]_{\mathcal{M}(D_\alpha(\Bn))},
    \end{equation*}
    which implies
    \begin{equation*}
        [g^{3k+1}] \subseteq [f].
    \end{equation*}
\end{proof}

\subsection{Peak Sets of $A^{\infty}(\Bn)$} \label{Subsec2.3}

\begin{defn}
    For $\zeta \in \partial \Bn$, the complex tangent space to the sphere is
    defined by
    \begin{equation*}
        T_\zeta^{\mathbb C} \partial \Bn
        := \{ v \in \mathbb C^n :  \langle v , \zeta \rangle = 0 \}.
    \end{equation*}    
\end{defn}

\begin{defn} \label{Defn: complex tangential}
Let $M \subset \mathbb \partial \Bn$ be a $\CC^k$ real submanifold.
We say that $M$ is \emph{complex tangential} if, for every $\zeta \in M$,
\begin{equation*}
    T_\zeta M \subset T_\zeta^{\mathbb C}\partial \Bn.
\end{equation*}
\end{defn}

\begin{defn}
    Let $M \subset \mathbb \partial \Bn$ be a $\CC^k$ real submanifold.
    We say that $M$ is \emph{totally real} if, for every $\zeta \in M$,
    \begin{equation*}
        T_\zeta M \cap i T_\zeta M = \{0\}.
    \end{equation*}
\end{defn}

\begin{defn} \label{Def: peak set}
     A closed subset \( E \subset \partial \Bn \) is called a \emph{peak set} for \( A^\infty(\Bn) \) if there exists a function \( f \in A^\infty(\Bn) \) such that
    \begin{equation*}
    f(z) = 1 \quad \text{for all } z \in E,
    \end{equation*}
    and
    \begin{equation*}
        |f(z)| < 1 \quad \text{for all } z \in \overline{\mathbb{B}}_n \setminus E.
    \end{equation*}
    Equivalently, \( E \) is a peak set if there exists a function \( g \in A^\infty(\Bn) \) such that
    \begin{equation*}
        g(z) = 0 \quad \text{for all } z \in E,
    \end{equation*}
    and
    \begin{equation*}
        \operatorname{Re} g(z) > 0 \quad \text{for all } z \in \overline{\mathbb{B}}_n \setminus E.
    \end{equation*}
\end{defn}

\begin{thm} [\cite{ChaumatChollet1979ensembles}, Theorem~21] \label{Thm: PeakSet Chaumat Chollet}
     Let $ M \subset \partial \Bn $ be a $ C^{\infty} $ complex tangential manifold of real dimension at most $n-1$. Then every compact subset $ K \subset M $ is a peak set for $ A^{\infty}(\Bn) $.
    
    Moreover, it follows from equation~(21.4) in \cite{ChaumatChollet1979ensembles}, together with the discussion on page~182 therein, that for a corresponding peak function \( g \) in the sense of Definition~\ref{Def: peak set}, there exist a constant \( c>0 \) and an open neighborhood \( U \subset \mathbb{C}^n \) of \( K \) such that
    \begin{equation} \label{Real Part inequality}
        \operatorname{Re} g(z) \geq c \cdot \operatorname{dist}(z, M)^2 \quad \text{for all } z \in \Bn \cap U.
    \end{equation}
\end{thm}

\begin{cor} \label{Peak_Corollary}
    Under the hypotheses of Theorem~\ref{Thm: PeakSet Chaumat Chollet}, for every point \( z_0 \in K \), there exist constants \( c > 0 \) and \( c' > 0 \), and an open neighborhood \( U \subset \mathbb{C}^n \) of \( z_0 \), such that for all \( z \in U \cap \Bn \), we have
    \begin{equation*}
        |g(z)| \geq c \left( u + |v| + \mathrm{dist}(z, M)^2 \right),
    \end{equation*}
    and
    \begin{equation*}
        |g(z)| \leq c'\, \mathrm{dist}(z, M),
    \end{equation*}
    where \( w = u + iv \) are local holomorphic coordinates centered at \( z_0 \), with \( u \) in the inward normal direction to \( \partial \Bn \), and \( v \) in the tangential direction and transverse to $M$ by the hypothesis on $M$.
\end{cor}
\begin{proof}
    Note that \( \operatorname{Re} g \) is a non-constant pluriharmonic function that attains its absolute minimum on the boundary point \( z_0 \). By Hopf's Lemma, we have
    \begin{equation*}
        \left. \frac{\partial}{\partial u} \operatorname{Re} g \right|_{z_0} \neq 0.
    \end{equation*}
    Applying the Cauchy–Riemann equations then yields
    \begin{equation*}
        \left. \frac{\partial}{\partial v} \operatorname{Im} g \right|_{z_0} = \left. \frac{\partial}{\partial u} \operatorname{Re} g \right|_{z_0} \neq 0.
    \end{equation*}
    This, together with \eqref{Real Part inequality}, completes the proof.
\end{proof}

\section{Proof of the Main Result} \label{Section3}

\begin{lem} \label{lem: Rineq}
        Let $g \in A^\infty(\Bn)$ be non-vanishing on $\Bn$ then for every positive integer $l$ there exists a constant $C_l>0$ independent of $r$ such that we have:
        \begin{equation}
            |R^l \frac{g}{g_r}| \leq  \frac{C_l(1-r)}{|g_r|^{l+1}}
        \end{equation}
    \end{lem}
\begin{proof}
    Let $\phi = g_r Rg - g Rg_r$. For an integer $0 \le t \le l$, we may write:
    \begin{equation} \label{eq: Rineq1}
    \begin{aligned}
        |R^t \phi|
        &= \left| R^t\left((g_r-g)Rg + g(Rg-Rg_r)\right) \right| \\
        &\le \left| R^t\left((g_r-g)Rg\right) \right|
        + \left| R^t\left(g(Rg-Rg_r)\right) \right| \\
        &\le \sum_{s=0}^t
        \left|R^s(g-g_r)\right|
        \left|R^{t+1-s}g\right|
        + \sum_{s=0}^t
        \left|R^s(Rg-Rg_r)\right|
        \left|R^{t-s}g\right|.
    \end{aligned}
    \end{equation}

    Since $g$ and all its derivatives are bounded, there exists $C>0$ such that
    \begin{equation} \label{eq: Rineq2}
        |R^s g|(z) \le C,
    \end{equation}
    for all $z \in \overline{\mathbb{B}}_n$ and all $s \in \{1,\ldots,l+1\}$. Also, for $s \in \{1,\ldots,l+1\}$,
    \begin{equation} \label{eq: Rineq3}
        |R^s(g-g_r)(z)|
        = |R^s g(z)-R^s g(rz)|
        \le \sup_{w \in \mathbb{B}_n} |\nabla (R^s g)(w)|\, |z-rz|
        \le C'(1-r),
    \end{equation}
    where $C'>0$ depends only on $l$ and $g$. From \eqref{eq: Rineq1}, \eqref{eq: Rineq2}, and \eqref{eq: Rineq3}, we infer that
        \begin{equation*}
            |R^t \phi| \le C'_l (1-r).
        \end{equation*}
    Noting that $|R^l \dfrac{g}{g_r}| = |R^{l-1} \dfrac{\phi}{g_r^2}|$ finishes the proof.
\end{proof}

\begin{lem} \label{cyc_g_lem}
    Let $M \subset \partial \Bn$, $K \subseteq M$, and $g$ be as in Theorem~\ref{Thm: PeakSet Chaumat Chollet}, with $m \coloneqq \dim_{\R} M \leq n-1$. Then $\mathcal{Z}(g) \cap \overline{\mathbb{B}}_n = K$, and $g$ is cyclic in $D_\alpha(\Bn)$ for any $\alpha \leq \alpha_c = \frac{2n-m}{2}$.
\end{lem}
\begin{proof}
    The fact that $\mathcal{Z}(g) \cap \overline{\mathbb{B}}_n = K$ comes directly from the fact that $g$ is a peak function for the set $K$. We will show that this function is cyclic in \( D_{\alpha_c}(\Bn) \). Note that this proves the cyclicity of $g$ in $D_{\alpha}(\Bn)$ for any $\alpha \leq \alpha_c$. Indeed, this follows from the fact that in this case $\|\cdot\|_{\alpha} \leq \|\cdot\|_{\alpha_c}$. Hence, if there exists a sequence of polynomials $(p_n)$ such that $\|1 - p_n g\|_{\alpha_c} \to 0$, then we also have $\|1 - p_n g\|_{\alpha} \to 0$.

    Let \( r \in (0,1) \), and define \( g_r(z) = g(rz) \). Our goal is to prove that 
    \begin{equation*} 
       \sup_{r \in (0,1)} \left\| \frac{g}{g_r} \right\|_{\alpha_c} < \infty,
    \end{equation*}
    so that we can invoke Theorem~\ref{Radial_Dilation} to conclude the proof.

    Let \( l \) be a sufficiently large positive integer such that \( 2l - \alpha_c > 0 \). Then, by Lemma~\ref{RLem}, we have:
    \begin{equation} \label{g/gr norm eq}
        \left\| \frac{g}{g_r} \right\|^2_{\alpha_c} \asymp \int_{\Bn} (1 - |z|^2)^{2l - \alpha_c - 1} \left| R^l \left( \frac{g}{g_r} \right) \right|^2 \, dv(z).
    \end{equation}

    Applying Lemma~\ref{lem: Rineq} to \eqref{g/gr norm eq}, we obtain:
    \begin{equation}
        \left\| \frac{g}{g_r} \right\|^2_{\alpha_c} \lesssim (1 - r)^2 \int_{\Bn} (1 - |z|^2)^{2l - \alpha_c - 1} \left| \frac{1}{g_r} \right|^{2l + 2} \, dv(z).
    \end{equation}

    Now, for $p \in M$, let $Z_p(x,y,w)$ be a system of local coordinates at $p$, defined on an open neighborhood $U_p$ of $p$, where $x \in \mathbb{R}^m$, $y \in \mathbb{R}^{2n-m-2}$, and $w = u+iv \in \mathbb{C}$, such that
    \[
        Z_p(x,0,0) \in M, \qquad
        Z_p(x,y,0) \in \partial \mathbb{B}_n,
    \]
    and where $w$ is as in Corollary~\ref{Peak_Corollary}.
    We may choose $0<r_p<1$ sufficiently large such that, for any $r_p \le r <1$, we have $rZ_p \in U_p$. Then there exists $c_p>0$ such that
    \begin{equation} \label{g/gr eq2}
        |g_r(Z_p)| \ge c_p \left( (1-r) + u + |v| + |y|^2 \right).
    \end{equation}

    Since $M$ is compact, we may cover it by finitely many such neighborhoods $U_p$. This allows us to estimate the integral in \eqref{g/gr norm eq} in the following form:
    \begin{equation} \label{g/gr eq3}
        \left\| \frac{g}{g_r} \right\|^2_{\alpha_c}
        \lesssim (1 - r)^2 \int_{(0,1)^m} \int_{(-1,1)^{2n - m - 2}} \int_{(-1,1)} \int_{(0,1)} 
        \frac{u^{2l - \alpha_c - 1}}{((1 - r) + u + |v| + |y|^2)^{2l + 2}} \, du\, dv\, dy\, dx.
    \end{equation}

    Letting \( u = u^\gamma \) where \( \gamma = 2l - \alpha_c \), we may rewrite \eqref{g/gr eq3} as
    \begin{align*} 
        \left\| \frac{g}{g_r} \right\|^2_{\alpha_c}
        &\lesssim (1 - r)^2 \int_{(0,1)^m} \int_{(-1,1)^{2n - m - 2}} \int_{(-1,1)} \int_{(0,1)} 
        \frac{1}{((1 - r) + u^{1/\gamma} + |v| + |y|^2)^{2l + 2}} \, du\, dv\, dy\, dx \\
        &\asymp (1 - r)^2 \int_{(0,1)^m} \int_{(-1,1)^{2n - m - 2}} \int_{(-1,1)} \int_{(0,1)} 
        \frac{1}{((1 - r)^\gamma + u + |v|^\gamma + |y|^{2\gamma})^{\frac{2l + 2}{\gamma}}} \, du\, dv\, dy\, dx \\
        &\asymp (1 - r)^2 \int_{(0,1)^m} \int_{(-1,1)^{2n - m - 2}} \int_{(-1,1)}
        \frac{1}{((1 - r)^\gamma + |v|^\gamma + |y|^{2\gamma})^{\frac{2l + 2}{\gamma} - 1}} \, dv\, dy\, dx \\
        &\asymp (1 - r)^2 \int_{(0,1)^m} \int_{(-1,1)^{2n - m - 2}} \int_{(-1,1)}
        \frac{1}{((1 - r) + |v| + |y|^2)^{\alpha_c + 2}} \, dv\, dy\, dx \\
        &\asymp (1 - r)^2 \int_{(0,1)^m} \int_{(-1,1)^{2n - m - 2}} \int_{(0,1)}
        \frac{1}{((1 - r) + v + |y|^2)^{\alpha_c + 2}} \, dv\, dy\, dx \\
        &\asymp (1 - r)^2 \int_{(0,1)^m} \int_{(-1,1)^{2n - m - 2}}
        \frac{1}{((1 - r) + |y|^2)^{\alpha_c + 1}} \, dy\, dx \\
        &= (1 - r)^2 \int_{(0,1)^m} \int_{(-1,1)^{2n - m - 2}}
        \frac{1}{((1 - r) + y_1^2 + \cdots + y_{2n - m - 2}^2)^{\alpha_c + 1}} \, dy\, dx \\
        &\asymp (1 - r)^2 \int_{(0,1)^m} \int_{(0,1)^{2n - m - 2}}
        \frac{1}{((1 - r)^{1/2} + y_1 + \cdots + y_{2n - m - 2})^{2n - m + 2}} \, dy\, dx \\
        &\asymp (1 - r)^2 \int_{(0,1)^m}
        \frac{1}{(1 - r)^{2}} \, dx \\
        &= 1.
    \end{align*}
\end{proof}

\begin{proof}[Proof of Theorem~\ref{MainThm}]
    From Theorem~\ref{Thm: PeakSet Chaumat Chollet}, we know that \( \mathcal{Z}(f) \cap \overline{\mathbb{B}}_n = M \) is a peak set. Therefore, by Lemma~\ref{cyc_g_lem}, there exists \( g \in A^\infty(\Bn) \) such that \( \mathcal{Z}(g) \cap \overline{\mathbb{B}}_n = M \), and \( g \) is a cyclic multiplier of \( D_{\frac{2n - m}{2}}(\Bn) \).

    From the construction of $g$, we obtain:
    \begin{equation*}
        |g(z)| \lesssim \operatorname{dist}(z, M).
    \end{equation*}

    On the other hand, Łojasiewicz’s inequality \cite{KrantzParks2002} gives that there exists a positive integer \( l \) and a constant \( C > 0 \) such that for all $z \in \overline{\mathbb{B}}_n$ we have:
    \begin{equation*}
        |f(z)| \geq C \operatorname{dist}(z, M)^l.
    \end{equation*}
    Therefore, there exists $c>0$ such that we have on $\Bn$:
    \begin{equation*}
        c |g^l(z)| \le |f(z)|.
    \end{equation*}
    Replacing $g$ by $c g^l$ and keeping the notation $g$, we may assume that $f$ and $g$ satisfy the assumptions of Theorem~\ref{Subspace_Inclusion}. Hence, there exists $N \in \mathbb{N}$ such that
    \begin{equation*}
        [g^N] \subseteq [f].
    \end{equation*}
    Finally, since $g$ is cyclic (and therefore so is $g^N$, as $g$ is a multiplier), the proof is complete.
\end{proof}

\section*{Acknowledgements}
The author would like to express sincere gratitude to Łukasz Kosiński and Pascal Thomas, the author’s PhD advisors, for their invaluable guidance and ongoing support throughout this work. The author is particularly grateful to Pascal Thomas for his substantial help in the development of the present idea, and to Łukasz Kosiński for his continued insightful advice and encouragement. Finally, the author acknowledges the Institut de Mathématiques de Toulouse and the Department of Mathematics at Jagiellonian University for providing a stimulating research environment.

\bibliographystyle{plain} 

\end{document}